\documentclass[11pt]{amsart}

\usepackage{amssymb}
\usepackage{mathrsfs}
\usepackage{amsfonts}
\usepackage{latexsym,amsmath,amsthm,amssymb,amsfonts}
\usepackage[usenames]{color}
\usepackage{xspace,colortbl}
\usepackage{graphicx}
\usepackage{tipa}

\newcommand{\be}{\begin{equation}}
\newcommand{\ee}{\end{equation}}
\newcommand{\beq}{\begin{eqnarray}}
\newcommand{\eeq}{\end{eqnarray}}

\newtheorem{prop}{Proposition}[section]

\newtheorem{remark}[prop]{Remark}

\def\begeq{\begin{equation}}
\def\endeq{\end{equation}}

\def\odot{\setbox0=\hbox{$\bigcirc$}\relax \mathbin {\hbox
to0pt{\raise.5pt\hbox to\wd0{\hfil $\wedge$\hfil}\hss}\box0 }}

\numberwithin{equation} {section}

\numberwithin{equation}{section}
\newtheorem{theorem}{\bf Theorem}[section]

\newtheorem{lemma}[theorem]{\bf Lemma}

\newtheorem{corollary}[theorem]{\bf Corollary}

\allowdisplaybreaks

\begin{document}

\title[inverse mean curvature flow for spacelike graphic hypersurfaces]
 {Inverse mean curvature flow for spacelike graphic hypersurfaces with boundary \\in Lorentz-Minkowski space $\mathbb{R}^{n+1}_{1}$}

\author{
 Ya Gao,~~ Jing Mao$^{\ast}$}

\address{
 Faculty of Mathematics and Statistics, Key Laboratory of
Applied Mathematics of Hubei Province, Hubei University, Wuhan
430062, China. }

\email{Echo-gaoya@outlook.com, jiner120@163.com}

\thanks{$\ast$ Corresponding author}

\date{}
\begin{abstract}
In this paper, we consider the evolution of spacelike graphic
hypersurfaces defined over a convex piece of hyperbolic plane
$\mathscr{H}^{n}(1)$, of center at origin and radius $1$, in the
$(n+1)$-dimensional Lorentz-Minkowski space $\mathbb{R}^{n+1}_{1}$
along the inverse mean curvature flow with the vanishing Neumann
boundary condition, and prove that this flow exists for all the
time. Moreover, we can show that, after suitable rescaling, the
evolving spacelike graphic hypersurfaces converge smoothly to a
piece of hyperbolic plane of center at origin and prescribed radius,
which actually corresponds to a constant function defined over the
piece of $\mathscr{H}^{n}(1)$, as time tends to infinity.
\end{abstract}

\maketitle {\it \small{{\bf Keywords}: Inverse mean curvature flow,
spacelike hypersurfaces, Lorentz-Minkowski space, Neumann boundary
condition.}

}

\section{Introduction}

The study of curvature flows is a very hot topic in Geometric
Analysis since the 80's of last century, and it attracts lots of
geometers' attention. There is an impressive event in the
development of the theory of curvature flows, that is, in 2002, G.
Perelman used the Hamilton-Ricci flow, an intrinsic curvature flow,
to successfully solve the geometrization conjecture, which directly
gives the solution to the famous $3$-dimensional Poincar\'{e}
conjecture.

There are two important extrinsic curvature flows -- the mean
curvature flow (MCF for short) and the inverse mean curvature flow
(IMCF for short). Given a prescribed ambient space, the so-called
MCF is actually deforming a one-parameter family of its submanifolds
along the unit inward normal vector with a speed equal to the mean
curvature $H$, while the IMCF is deforming a one-parameter family of
its submanifolds along the unit outward normal vector with a speed
equal to reciprocal of the mean curvature $1/H$. Hence, generally,
the MCF is a contracting flow while the IMCF is an expanding flow.

To our knowledge, the start of the study of MCF maybe is due to
Brakke \cite{kab} where he used the geometric measure theory to
investigate the motion of surface by its mean curvature, while
Huisken \cite{gh1} gave a pioneering contribution to this theory --
by mainly using the method of $L^{p}$ estimates, he proved that if
the initial hypersurface $\mathcal{W}^{n}$ is a compact strictly
convex hypersurface in the $(n+1)$-dimensional ($n\geq2$) Euclidean
space $\mathbb{R}^{n+1}$, the MCF has a unique smooth solution on
the finite time interval $[0,T_{\max})$ with $T_{\max}<\infty$, and
the evolving hypersurfaces $\mathcal{W}^{n}_{t}$ contract to a
single point as $t\rightarrow T_{\max}$. Moreover, after an
area-preserving rescaling, the rescaled hypersurfaces converge in
$C^{\infty}$-topology to a round sphere having the same area as
$\mathcal{W}^{n}$. Later, this classical result was successfully
generalized by himself to the situation that the ambient space
$\mathbb{R}^{n+1}$ was replaced by smooth complete Riemannian
manifolds whose injectivity radius, the curvature tensor and the
second fundamental form satisfy some pinching assumptions (see
\cite[Theorem 1.1]{gh2} for details). There are some interesting
curvature flows which are closely related to the classical MCF and
have been investigated already. For instance, the volume-preserving
MCF (see Huisken's work \cite{Huisken 1987}), the area-preserving
MCF (see McCoy's work \cite{James Mccoy 2003}), the mixed-volume
preserving MCF (see McCoy's work \cite{James Mccoy 2004}), the MCF
of entire graphs (see Ecker-Huisken's work \cite{eh}), the graphic
MCF of arbitrary codimension (see, e.g., Wang's work \cite{mtw}),
and so on. The corresponding author here, Prof. J. Mao, also has
some works on the MCF (see, e.g., \cite{chmx,lmw}). Specially, in
\cite{chmx}, for a given $3$-dimensional Lorentz manifold
$\mathcal{W}^{2}\times\mathbb{R}$ with the metric $
\sum_{i,j=1}^{2}\rho_{ij}dw^{i}\otimes dw^{j}-ds\otimes ds$, where
$\mathcal{W}^{2}$ is a $2$-dimensional complete Riemannian manifold
with the metric $\sum_{i,j=1}^{2}\rho_{ij}dw^{i}\otimes dw^{j}$ and
nonnegative Gaussian curvature, the authors therein investigated the
evolution of spacelike graphs (defined over compact, strictly convex
domains in $\mathcal{W}^{2}$) along the nonparametric MCF with
prescribed nonzero Neumann boundary condition (NBC for short), and
proved that this flow exists for all the time and its solutions
converge to ones moving only by translation. This interesting
conclusion somehow extends, for instance, the following results:
\begin{itemize}

\item (Huisken \cite{gh3}) Graphs defined over
bounded domains (with $C^{2,\gamma}$ boundary) in $\mathbb{R}^{n}$
($n\geq2$), which are evolving by the MCF with vertical contact
angle boundary condition (i.e., vanishing NBC), have been
investigated, and it was proven that this evolution exists for all
the time and the evolving graphs converge to a constant function as
time tends to infinity (i.e., $t\rightarrow\infty$);

\item (Altschuler-Wu \cite{aw})  Graphs, defined over strictly convex
compact domains in $\mathbb{R}^{2}$, evolved by the non-parametric
MCF with prescribed contact angle (not necessary to be vertical,
i.e., the NBC is not necessary to be zero), converge to translating
surfaces as $t\rightarrow\infty$.

\end{itemize}
Except \cite{aw,chmx,gh3}, there are many other interesting results
about the MCF with boundary conditions in both Riemannian and
pseudo-Riemannian manifolds -- see, e.g.,
\cite{add-x1,add-x2,add-x3,add-x4,add-x5} and reference therein. The
MCF theory also has some important applications. For instance, the
curve shortening flow (i.e., the lower dimensional case of the MCF)
can be used in the image processing (see, e.g., \cite{cf}); Topping
\cite{pt} used the MCF to successfully obtain some geometric
inequalities on surfaces with prescribed Gaussian curvature
assumptions.

As we know, Gerhardt \cite{Ge90} (or Urbas \cite{Ur}) firstly
considered the evolution of compact, star-shaped $C^{2,\gamma}$
hypersurfaces $\mathcal{W}_{0}^{n}$ in $\mathbb{R}^{n+1}$ along the
flow equation
\begin{eqnarray} \label{ICF-1}
\frac{\partial}{\partial t}X=\frac{1}{F}\nu,
\end{eqnarray}
where $F$ is a positive, symmetric, monotone, homogeneous of degree
one, concave function w.r.t. principal curvatures of the evolving
hypersurfaces
$\mathcal{W}^{n}_{t}=X(\mathbb{S}^{n},t)=X_{t}(\mathbb{S}^{n})$,
\footnote{In this paper, $\mathbb{S}^{n}$ stands for the unit
Euclidean $n$-sphere.}and $\nu$ is the outward unit normal vector of
$\mathcal{W}^{n}_{t}$. They separately proved that the flow exists
for all the time, and, after suitable rescaling, converge
exponentially fast to a uniquely determined sphere of prescribed
radius. This flow is called inverse curvature flow (ICF for short),
and clearly, $F=H$ is allowed and in this setting, the flow equation
becomes $\partial X/\partial t=\nu/H$, which is exactly the
classical IMCF equation.

The reason why geometers are interested in the study of the theory
of ICFs is that it has important applications in Physics and
Mathematics. For instance, by defining a notion of weak solutions to
IMCF, Huisken and Ilmanen \cite{hi1,hi2} proved the Riemannian
Penrose inequality by using the IMCF approach, which makes an
important step to possibly and completely solve the famous Penrose
conjecture in the General Relativity.  Also using the method of
IMCF, Brendle, Hung and Wang \cite{bhw} proved a sharp Minkowski
inequality for mean convex and star-shaped hypersurfaces in the
$n$-dimensional ($n\geq3$) anti-de Sitter-Schwarzschild manifold,
which generalized the related conclusions in the Euclidean
$n$-space. Besides, applying ICFs, Alexandrov-Fenchel type and other
types inequalities in space forms and even in some warped products
can be obtained - see, e.g., \cite{Gw1,Gw3,Li14,Li16,Mak}.

J. Mao also has some works on ICFs (see, e.g.,
\cite{Ch16,cmxx,cmtw,chmw,mt}), and we would like to give a brief
introduction to some of them such that readers can realize our
motivation of writing this paper clearly. In fact, Mao and his
collaborators have obtained:
\begin{itemize}
\item (\cite{Ch16}) The conclusion about the long-time existence and the convergence of IMCF in the anti-de Sitter-Schwarzschild $n$-manifold ($n\geq3$) obtained
 in \cite{bhw} was
successfully improved to a more general ICF (\ref{ICF-1}).

\item (\cite{cmtw}) The following ICF
\begin{equation}  \label{ICF-origin}
\left\{
\begin{aligned}
&\frac{\partial }{\partial t}X=\frac{1}{|X|^{\alpha}H(X)}\nu\\
&X(\cdot,0)=\mathcal{W}_{0}
\end{aligned}
\right.
\end{equation}
has been considered, where $\alpha\geq0$, $\mathcal{W}_{0}$ is a
closed, star-shaped and strictly mean convex
$C^{2,\gamma}$-hypersurface ($0<\gamma<1$) in $\mathbb{R}^{n+1}$,
$X(\cdot,t):\mathbb{S}^{n}\rightarrow\mathbb{R}^{n+1}$ is a
one-parameter family of hypersurfaces immersed into
$\mathbb{R}^{n+1}$ with $\mathcal{W}_{t}=X(\mathbb{S}^{n},t)$, $|X|$
is the distance from the point $X(x, t)$ to the origin of
$\mathbb{R}^{n+1}$, and, as before, $\nu$ is the unit outward normal
vector of $\mathcal{W}_t$. It has been proven that the flow
(\ref{ICF-origin}) exists for all the time and, after rescaling, the
evolving hypersurfaces converge smoothly to a round sphere. Clearly,
if $\alpha=0$, the ICF (\ref{ICF-origin}) degenerates into the
classical IMCF in $\mathbb{R}^{n+1}$ directly. This interesting
result tells us that if the homogeneous anisotropic factor
$|X|^{-\alpha}$ was imposed to the classical IMCF for compact,
star-shaped and strictly mean convex hypersurfaces in
$\mathbb{R}^{n+1}$, there is almost no affect on conclusions about
the long-time existence and the convergence of the flow equation.

\item (\cite{mt}) Given a smooth convex cone in $\mathbb{R}^{n+1}$
($n\geq2$),  the authors therein considered strictly mean convex
hypersurfaces with boundary which are star-shaped with respect to
the center of the cone and which meet the cone perpendicularly. If
those hypersurfaces inside the cone evolve along the ICF $\partial X
/\partial t=\nu/|X|^{\alpha}H(X)$, $\alpha\geq0$, then, by using the
convexity of the cone in the derivation of the gradient and
H\"{o}lder estimates, they can prove that this evolution exists for
all the time and the evolving hypersurfaces converge smoothly to a
piece of a round sphere as time tends to infinity. Clearly, this
interesting conclusion covers Marquardt's result \cite[Theorem
1]{Mar} (corresponding to $\alpha=0$) as a special case.
\end{itemize}

Our successful experience on studying MCF of spacelike surfaces in
Lorentz $3$-manifolds (see \cite{chmx}) encourages us to consider
the following problem:
\begin{itemize}
\item \emph{Whether one can consider ICFs in Lorentz manifolds or
not?}
\end{itemize}
We are able to give a positive answer to the above problem in this
paper. In order to state our main conclusion clearly, we need to
give several notions first.

Throughout this paper, let $\mathbb{R}^{n+1}_{1}$ be the
$(n+1)$-dimensional ($n\geq2$) Lorentz-Minkowski space with the
following Lorentzian metric
\begin{eqnarray*}
\langle\cdot,\cdot\rangle_{L}=dx_{1}^{2}+dx_{2}^{2}+\cdots+dx_{n}^{2}-dx_{n+1}^{2}.
\end{eqnarray*}
In fact, $\mathbb{R}^{n+1}_{1}$ is an $(n+1)$-dimensional Lorentz
manifold with index $1$. Denote by
\begin{eqnarray*}
\mathscr{H}^{n}(1)=\{(x_{1},x_{2},\cdots,x_{n+1})\in\mathbb{R}^{n+1}_{1}|x_{1}^{2}+x_{2}^{2}+\cdots+x_{n}^{2}-x_{n+1}^{2}=-1~\mathrm{and}~x_{n+1}>0\},
\end{eqnarray*}
which is exactly the hyperbolic plane\footnote{The reason why we
call $\mathscr{H}^{n}(1)$ a hyperbolic plane is that it is a
simply-connected Riemannian $n$-manifold with constant negative
curvature and is geodesically complete.} of center $(0,0,\ldots,0)$
(i.e., the origin of $\mathbb{R}^{n+1}$) and radius $1$ in
$\mathbb{R}^{n+1}_{1}$. Clearly, from the Euclidean viewpoint,
$\mathscr{H}^{2}(1)$ is one component of a hyperboloid of two
sheets.

In this paper, we consider the evolution of spacelike graphs
(contained in a prescribed convex domain) along the IMCF with zero
NBC, and can prove the following main conclusion.

\begin{theorem}\label{main1.1}
Let $M^n\subset\mathscr{H}^{n}(1)$ be some convex piece of the
hyperbolic plane $\mathscr{H}^{n}(1)\subset\mathbb{R}^{n+1}_{1}$,
and $\Sigma^n:=\{rx\in \mathbb{R}^{n+1}_{1}| r>0, x\in
\partial M^n\}$. Let
$X_{0}:M^{n}\rightarrow\mathbb{R}^{n+1}_{1}$ such that
$M_{0}^{n}:=X_{0}(M^{n})$ is a compact, strictly mean convex
spacelike $C^{2,\gamma}$-hypersurface ($0<\gamma<1$) which can be
written as a graph over $M^n$.
 Assume that
 \begin{eqnarray*}
M_{0}^{n}=\mathrm{graph}_{M^n}u_{0}
 \end{eqnarray*}
 is a graph over $M^n$ for a positive
map $u_0: M^n\rightarrow \mathbb{R}$ and
 \begin{eqnarray*}
\partial M_{0}^{n}\subset \Sigma^n, \qquad
\langle\mu\circ X_{0}, \nu_0\circ X_{0} \rangle_{L}|_{\partial
M^n}=0,
 \end{eqnarray*}
 where $\nu_0$ is the past-directed timelike unit normal vector of $M_{0}^{n}$, $\mu$ is a spacelike vector field
defined along $\Sigma^{n}\cap \partial M^{n}=\partial M^{n}$
satisfying the following property:
\begin{itemize}
\item For any $x\in\partial M^{n}$, $\mu(x)\in T_{x}M^{n}$, $\mu(x)\notin T_{x}\partial
M^{n}$, and moreover\footnote{As usual, $T_{x}M^{n}$, $T_{x}\partial
M^{n}$ denote the tangent spaces (at $x$) of $M^{n}$ and $\partial
M^{n}$, respectively. In fact, by the definition of $\Sigma^{n}$
(i.e., a time cone), it is easy to have $\Sigma^{n}\cap\partial
M^{n}=\partial M^{n}$, and we insist on writing as
$\Sigma^{n}\cap\partial M^{n}$ here is just to emphasize the
relation between $\Sigma^{n}$ and $\mu$. Since $\mu$ is a vector
field defined along $\partial M^{n}$, which satisfies $\mu(x)\in
T_{x}M^{n}$, $\mu(x)\notin T_{x}\partial M^{n}$ for any
$x\in\partial M^{n}$, together with the construction of
$\Sigma^{n}$, it is feasible to require $\mu(x)=\mu(rx)$. The
requirement $\mu(x)=\mu(rx)$ makes the assumptions $\langle\mu\circ
X_{0}, \nu_0\circ X_{0} \rangle_{L}|_{\partial M^n}=0$, $\langle \mu
\circ X, \nu\circ X\rangle_{L}=0$ on $\partial M^n \times(0,\infty)$
are reasonable, which can be seen from Lemma \ref{lemma2-1} below in
details. Besides, since $\nu$ is timelike, the vanishing Lorentzian
inner product assumptions on $\mu,\nu$ implies that $\mu$ is
spacelike.}, $\mu(x)=\mu(rx)$.
\end{itemize}
Then we have:

(i) There exists a family of strictly mean convex spacelike
hypersurfaces $M_{t}^{n}$ given by the unique embedding
\begin{eqnarray*}
X\in C^{2+\gamma,1+\frac{\gamma}{2}} (M^n\times [0,\infty),
\mathbb{R}^{n+1}_{1}) \cap C^{\infty} (M^n\times (0,\infty),
\mathbb{R}^{n+1}_{1})
\end{eqnarray*}
with $X(\partial M^n, t) \subset \Sigma^n$ for $t\geq 0$, satisfying the following system
\begin{equation}\label{Eq}
\left\{
\begin{aligned}
&\frac{\partial }{\partial t}X=\frac{1}{H}\nu ~~&&in~
M^n \times(0,\infty)\\
&\langle \mu \circ X, \nu\circ X\rangle_{L}=0~~&&on~ \partial M^n \times(0,\infty)\\
&X(\cdot,0)=M_{0}^{n}  ~~&& in~M^n
\end{aligned}
\right.
\end{equation}
where $H$ is the mean curvature of
$M_{t}^{n}:=X(M^n,t)=X_{t}(M^{n})$, and $\nu$ is the past-directed
timelike unit normal vector\footnote{On \cite[page 299]{chmx}, the
authors therein gave the expression of ``downward" unit normal vector
of spacelike graphs $\mathcal{G}$. Here, we would like to clarify
that the more precise terminology therein should be
``\emph{past-directed unit normal vector}", which, of course, is a
timelike vector.} of $M_{t}^{n}$. Moreover, the H\"{o}lder norm on
the parabolic space $M^n\times(0,\infty)$ is defined in the usual
way (see, e.g., \cite[Note 2.5.4]{Ge3}).

(ii) The leaves $M_{t}^{n}$ are spacelike graphs over $M^n$, i.e.,
 \begin{eqnarray*}
M_{t}^{n}=\mathrm{graph}_{M^n}u(\cdot, t).
\end{eqnarray*}

(iii) Moreover, the evolving spacelike hypersurfaces converge
smoothly after rescaling to a piece of
$\mathscr{H}^{n}(r_{\infty})$, where $r_{\infty}$ satisfies
\begin{eqnarray*}
\frac{1}{\sup\limits_{M^{n}}u_{0}}\left(\frac{\mathcal{H}^n(M_{0}^{n})}{\mathcal{H}^n(M^{n})}\right)^{\frac{1}{n}}\leq
r_{\infty}
\leq\frac{1}{\inf\limits_{M^{n}}u_{0}}\left(\frac{\mathcal{H}^n(M_{0}^{n})}{\mathcal{H}^n(M^{n})}\right)^{\frac{1}{n}},
\end{eqnarray*}
where $\mathcal{H}^n(\cdot)$ stands for the $n$-dimensional
Hausdorff measure of a prescribed Riemannian $n$-manifold,
$\mathscr{H}^{n}(r_{\infty}):=\left\{r_{\infty}x\in\mathbb{R}^{n+1}_{1}|
x\in\mathscr{H}^{n}(1)\right\}$.
\end{theorem}

\begin{remark}
\rm{ (1) In fact, $M^{n}$ is some \emph{convex} piece of the
spacelike hypersurface $\mathscr{H}^{n}(1)$ implies that the second
fundamental form of $\partial M^{n}$ is positive definite w.r.t. the
vector field $\mu$ (provided its direction is suitably chosen).
  \\
(2) For Theorem \ref{main1.1}, if the flow equation was replaced by
\begin{eqnarray} \label{1-4}
\frac{\partial }{\partial t}X=\frac{1}{|\langle
X,X\rangle_{L}|^{\alpha} H}\nu, \qquad (\alpha\leq0)
\end{eqnarray}
in $\mathbb{R}^{n+1}_{1}$ (or the flow equation, the ambient space
$\mathbb{R}^{n+1}_{1}$  were replaced by (\ref{1-4}) and the
$(n+1)$-dimensional Lorentz manifold
$\mathcal{W}^{n}\times\mathbb{R}$ respectively, where
$\mathcal{W}^{n}$ is a complete Riemannian $n$-manifold with
suitable Ricci curvature constraint), interesting conclusions can be
expected (see \cite{gm1,gm2} for details).\\
 (3) In the previous version of this paper (see arXiv:2104.10600v4),
 one might find that we have used a slightly different definition (for the second fundamental form) with the one given in our previous works \cite{chmx,GaoY} (i.e., having an
 exactly opposite sign with the formula (\ref{add-xxx}) below for calculating components of the second fundamental
 form). This difference leads to the situation that the mean
 curvature $H$ of spacelike graphic hypersurfaces computed therein
 has an opposite sign with the one used L\'{o}pez's setting in
 \cite{rl} (also would be different from the one used the setting in
 \cite{chmx,GaoY}). However, as explained clearly in Remark 1.1 of the previous
 version, if different orientations for the timelike unit normal
 vector $\nu$ have been used, then
 there would have no essential difference to deal with the
 system (\ref{Eq}), i.e., the IMCF with zero NBC in
 $\mathbb{R}^{n+1}_{1}$, in different settings for computing components of
 the second fundamental form. After we posted the previous version
 and some related works
 on arXiv, some (domestic and foreign) colleagues suggested that to
 avoid any potential confusion, it might be better to use a unified setting.
 Based on this reason, in this latest version, we decide to use the formula
 (\ref{add-xxx}) below -- equivalently, the setting for calculating components of the second fundamental
 form given in
 \cite{chmx,GaoY}.
 \\
 (4) As mentioned before,  in $\mathbb{R}^{n+1}$ ($n\geq2$), Marquardt \cite[Theorem
1]{Mar} firstly considered the evolution of strictly mean convex,
star-shaped hypersurface (w.r.t. the center of an open smooth convex
cone), which is contained in the cone and which meets the cone
perpendicularly, along the IMCF, and obtained the long-time
existence and the asymptotical behavior of the flow. Clearly,
Theorem \ref{main1.1} can be seen as the Minkowski-type version of
Marquardt's result \cite[Theorem 1]{Mar}. In the sequel, some
analytical techniques in \cite{Mar} (especially the part of
H\"{o}lder estimates) have been used for reference, but they should
be modification of the corresponding classical theory of
second-order nonlinear parabolic PDEs. Besides, one can check
interesting examples in Section \ref{se7} and will see the
\textbf{\emph{huge}} difference between Marquardt's result
\cite[Theorem 1]{Mar} and its Minkowski-type version here.
 }
\end{remark}

This paper is organized as follows. In Section \ref{se2}, several
fundamental structure equations of spacelike hypersurfaces in
$\mathbb{R}^{n+1}_{1}$ will be introduced. In Section \ref{se3}, we
will show that using the spacelike graphic assumption, the flow
equation (which generally is a system of PDEs) changes into a single
scalar second-order parabolic PDE. In Section \ref{se4}, several
estimates, including $C^0$, time-derivative and gradient estimates,
of solutions to the flow equation will be shown in details.
Estimates of higher-order derivatives of solutions to the flow
equation, which naturally leads to the long-time existence of the
flow, will be investigated in Section \ref{se5}. We will clearly
show the convergence of the rescaled flow in Section \ref{se6}. In
the end, two interesting and special examples will be shown in
Section \ref{se7}.

\section{The geometry of spacelike hypersurfaces in
$\mathbb{R}^{n+1}_{1}$} \label{se2}

As shown in \cite[Section 2]{GaoY}, we know the following fact:

\textbf{FACT}. Given an $(n+1)$-dimensional Lorentz manifold
$(\overline{M}^{n+1},\overline{g})$, with the metric $\overline{g}$,
and its spacelike hypersurface $M^{n}$. For any $p\in M^{n}$, one
can choose a local  Lorentzian orthonormal frame field
$\{e_{0},e_{1},e_{2},\ldots,e_{n}\}$ around $p$ such that,
restricted to $M^{n}$, $e_{1},e_{2},\ldots,e_{n}$ form orthonormal
frames tangent to $M^{n}$. Taking the dual coframe fields
$\{w_{0},w_{1},w_{2},\ldots,w_{n}\}$ such that the Lorentzian metric
$\overline{g}$ can be written as
$\overline{g}=-w_{0}^{2}+\sum_{i=1}^{n}w_{i}^{2}$. Making the
convention on the range of indices
\begin{eqnarray*}
0\leq I,J,K,\ldots\leq n; \qquad\qquad 1\leq i,j,k\ldots\leq n,
\end{eqnarray*}
and doing differentials to forms $w_{I}$, one can easily get the
following structure equations
\begin{eqnarray}
&&(\mathrm{Gauss~ equation})\qquad \qquad R_{ijkl}=\overline{R}_{ijkl}-(h_{ik}h_{jl}-h_{il}h_{jk}), \label{Gauss}\\
&&(\mathrm{Codazzi~ equation})\qquad \qquad h_{ij,k}-h_{ik,j}=\overline{R}_{0ijk},  \label{Codazzi}\\
&&(\mathrm{Ricci~ identity})\qquad \qquad
h_{ij,kl}-h_{ij,lk}=\sum\limits_{m=1}^{n}h_{mj}R_{mikl}+\sum\limits_{m=1}^{n}h_{im}R_{mjkl},
\label{Ricci}
\end{eqnarray}
and the Laplacian of the second fundamental form $h_{ij}$ of $M^{n}$
as follows
\begin{eqnarray} \label{LF}
&&\Delta
h_{ij}=\sum\limits_{k=1}^{n}\left(h_{kk,ij}+\overline{R}_{0kik,j}+\overline{R}_{0ijk,k}\right)+
\sum\limits_{k=1}^{n}\left(h_{kk}\overline{R}_{0ij0}+h_{ij}\overline{R}_{0k0k}\right)+\nonumber\\
&&\qquad\qquad
\sum\limits_{m,k=1}^{n}\left(h_{mj}\overline{R}_{mkik}+2h_{mk}\overline{R}_{mijk}+h_{mi}\overline{R}_{mkjk}\right)\nonumber\\
&&\qquad\quad
-\sum\limits_{m,k=1}^{n}\left(h_{mi}h_{mj}h_{kk}+h_{km}h_{mj}h_{ik}-h_{km}h_{mk}h_{ij}-h_{mi}h_{mk}h_{kj}\right),
\end{eqnarray}
where $R$ and $\overline{R}$ are the curvature tensors of $M^{n}$
and $\overline{M}^{n+1}$ respectively, $A:=h_{ij}w_{i}w_{j}$ is the
second fundamental form with $h_{ij}$ the coefficient components of
the tensor $A$, $\Delta$ is the Laplacian on the hypersurface
$M^{n}$, and, as usual, the comma ``," in subscript of a given
tensor means doing covariant derivatives. For detailed derivation of
the above formulae, we refer readers to, e.g., \cite[Section
2]{hzl}.

\begin{remark}
\rm{ There is one thing we prefer to mention here, that is, by using
the symmetry of the second fundamental form, we have
$\sum_{m,k=1}^{n}\left(h_{km}h_{mj}h_{ik}-h_{mi}h_{mk}h_{kj}\right)=0$,
which implies that the last term in the RHS of (\ref{LF}) becomes
\begin{eqnarray*}
-\sum\limits_{m,k=1}^{n}\left(h_{mi}h_{mj}h_{kk}-h_{km}h_{mk}h_{ij}\right).
\end{eqnarray*}
Here we insist on writing the Laplacian of $h_{ij}$ as (\ref{LF}) in
order to emphasize the origin of this formula.
 }
\end{remark}

Clearly, in our setting here, all formulae mentioned above can be
used directly with $\overline{M}^{n+1}=\mathbb{R}^{n+1}_{1}$ and
$\overline{g}=\langle\cdot,\cdot\rangle_{L}$.

For convenience, in the sequel we will use the Einstein summation
convention -- repeated superscripts and subscripts should be made
summation from $1$ to $n$. Given an $n$-dimensional Riemannian
manifold $M^{n}$ with the metric $g$, denote by
$\{y^{i}\}_{i=1}^{n}$ the local coordinate of $M^{n}$, and
$\frac{\partial}{\partial y^{i}}$, $i=1,2,\cdots,n$, the
corresponding coordinate vector fields ($\partial_{i}$ for short).
The Riemannian curvature $(1,3)$-tensor $R$ of $M^{n}$ can be
defined by
\begin{equation*}
R(X, Y)Z=-\nabla_{X}\nabla_{Y}Z+\nabla_{Y}\nabla_{X}Z+\nabla_{[X,
Y]}Z,
\end{equation*}
where $X,Y,Z\in\mathscr{X}(M)$ are tangent vector fields in the
tangent bundle $\mathscr{X}(M)$ of $M^{n}$, $\nabla$ is the gradient
operator on $M^{n}$, and, as usual, $[\cdot,\cdot]$ stands for the
Lie bracket. The component of the curvature tensor $R$ is defined by
\begin{equation*}
R\bigg({\frac{\partial}{\partial y^{i}}}, {\frac{\partial}{\partial
y^{j}}}\bigg){\frac{\partial}{\partial y^{k}}} \doteq
R_{ijk}^{l}{\frac{\partial}{\partial y^{l}}}
\end{equation*}
and $R_{ijkl}:= g_{mi}R_{jkl}^{m}$. Now, let us go back to our
setting -- the evolution of strictly mean convex spacelike graphs in
$\mathbb{R}^{n+1}_{1}$ along the IMCF with zero NBC. The second
fundamental form of the hypersurface $M_{t}^{n}=X(M^{n},t)$ w.r.t.
$\nu$ is given by
\begin{eqnarray} \label{add-xxx}
h_{ij}=\left\langle X_{,ij}, \nu\right\rangle_{L},
 \end{eqnarray}
where $\langle\nu,\nu\rangle_{L}=-1$, $X_{,ij}:=\partial_i
\partial_j X-\Gamma_{ij}^{k}X_k$ with $\Gamma_{ij}^{k}$ the
Christoffel symbols of the metric on $M^{n}_{t}$. Here we would like
to emphasize one thing, that is,
$X_{k}=(X_{t})_{\ast}(\partial_{k})$ with $(X_{t})_{\ast}$ the
tangential mapping induced by the map $X_{t}$. It is easy to have
the following identities
\begin{equation}\label{Gauss for}
X_{,ij}=-h_{ij}\nu, \qquad (\mathrm{Gauss~formula})
\end{equation}
\begin{equation}\label{Wein for}
\nu_{,i}=-h_{ij}X^j, \qquad (\mathrm{Weingarten~formula})
\end{equation}
Besides, using (\ref{Gauss}), (\ref{Codazzi}) and (\ref{LF}) with
the fact $\overline{R}=0$ in our setting, we have\footnote{~Clearly,
in the formula (\ref{Codazzi-1}), $\nabla_{k}h_{ij}$ should be
$\nabla_{\partial_{k}}h_{ij}$. In the sequel, without confusion and
if needed, we prefer to simplify covariant derivatives like this.}
\begin{equation}\label{Gauss-1}
R_{ijkl}=h_{il}h_{jk}-h_{ik}h_{jl},
\end{equation}
\begin{equation}\label{Codazzi-1}
\nabla_{k}h_{ij}=\nabla_{j}h_{ik}, \qquad (i.e.,~h_{ij,k}=h_{ik,j})
\end{equation}
and
\begin{eqnarray}\label{Laplace}
\Delta h_{ij}=H_{,ij}-H h_{ik}h^{k}_{j}+h_{ij}|A|^2.
\end{eqnarray}

We make an agreement that, for simplicity, in the sequel the comma
``," in subscripts will be omitted unless necessary.

\section{The scalar version of the flow equation} \label{se3}

Since the spacelike $C^{2,\gamma}$-hypersurface $M^{n}_{0}$ can be
written as a graph of $M^{n}\subset\mathscr{H}^{n}(1)$, there exists
a function $u_0\in C^{2,\gamma} (M^{n})$ such that
 $X_0: M^{n} \rightarrow \mathbb{R}^{n+1}_{1}$ has the form $x \mapsto
G_{0}:=(x,u_0(x))$. The hypersurface $M_{t}^{n}$ given by the
embedding
\begin{eqnarray*}
X(\cdot, t): M^{n}\rightarrow \mathbb{R}^{n+1}_{1}
\end{eqnarray*}
at time $t$ may be represented as a graph over $M^n\subset
\mathscr{H}^{n}(1)$, and then we can make ansatz
\begin{eqnarray*}
X(x,t)=\left(x,u(x,t)\right)
\end{eqnarray*}
for some function $u: M^{n} \times [0,T) \rightarrow \mathbb{R}$.
The following formulae are needed.

\begin{lemma} \label{lemma2-1}
Define $p:=X(x,t)$ and assume that a point on $\mathscr{H}^{n}(1)$
is described by local coordinates $\xi^{1},\ldots,\xi^{n}$, that is,
$x=x(\xi^{1},\ldots,\xi^{n})$. By the abuse of notations, let
$\partial_i$ be the corresponding coordinate fields on
$\mathscr{H}^{n}(1)$ and
$\sigma_{ij}=g_{\mathscr{H}^{n}(1)}(\partial_i,\partial_j)$ be the
Riemannian metric on $\mathscr{H}^{n}(1)$. Of course,
$\{\sigma_{ij}\}_{i,j=1,2,\ldots,n}$ is also the metric on
$M^{n}\subset\mathscr{H}^{n}(1)$. Following the agreement before,
denote by $u_{i}:=D_{i}u$, $u_{ij}:=D_{j}D_{i}u$, and
$u_{ijk}:=D_{k}D_{j}D_{i}u$ the covariant derivatives of $u$ w.r.t.
the metric $g_{\mathscr{H}^{n}(1)}$, where $D$ is the covariant
connection on $\mathscr{H}^{n}(1)$. Let $\nabla$ be the Levi-Civita
connection of $M_{t}^{n}$ w.r.t. the metric
$g:=u^{2}g_{\mathscr{H}^{n}(1)}-dr^{2}$ induced from the Lorentzian
metric $\langle\cdot,\cdot\rangle_{L}$ of $\mathbb{R}^{n+1}_{1}$.
Then, the following formulae hold:

(i) The tangential vector on $M_{t}^{n}$ is
\begin{eqnarray*}
X_{i}=\partial_{i}+u_i\partial_{r},
\end{eqnarray*}
and the corresponding past-directed timelike unit normal vector is
given by
\begin{eqnarray*}
\nu=-\frac{1}{v}\left(\partial_r+\frac{1}{u^2}u^j\partial_j\right),
\end{eqnarray*}
where $u^{j}:=\sigma^{ij}u_{i}$, and $v:=\sqrt{1-u^{-2}|D u|^2}$
with $D u$ the gradient of $u$.

(ii) The induced metric $g$ on $M_{t}^{n}$ has the form
\begin{equation*}
g_{ij}=u^2\sigma_{ij}-u_{i} u_{j},
\end{equation*}
and its inverse is given by
\begin{equation*}
g^{ij}=\frac{1}{u^2}\left(\sigma^{ij}+\frac{u^i  u^j
}{u^2v^{2}}\right).
\end{equation*}

(iii) The second fundamental form of $M_{t}^{n}$ is given by
\begin{eqnarray*}
h_{ij}=-\frac{1}{v}\left(\frac{2}{u}{u_i  u_j }-u_{ij} -u
\sigma_{ij}\right),
\end{eqnarray*}
and
\begin{eqnarray*}
h^i_j=g^{ik}h_{jk}=\left(\frac{1}{u v}\delta^i_j+\frac{1}{u
v}\widetilde{\sigma}^{ik}\varphi_{jk}\right),
\qquad\widetilde{\sigma}^{ij}=\sigma^{ij}+\frac{\varphi^i
\varphi^j}{v^2}.
 \end{eqnarray*}
Naturally, the mean curvature is given by
\begin{eqnarray*}
H=\sum_{i=1}^{n}h^i_i=\frac{1}{u
v}\bigg(n+(\sigma^{ij}+\frac{\varphi^i\varphi^j}{v^{2}})\varphi_{ij}\bigg).
\end{eqnarray*}
where $\varphi=\log u$.

(iv) Let $p=X(x,t)\in \Sigma^n$ with $x\in\partial M^{n}$,
$\hat{\mu}(p)\in T_{p}M^{n}_{t}$, $\hat{\mu}(p)\notin T_{p}\partial
M^{n}_{t}$, $\mu=\mu^i(x)
\partial_{i}(x)$ at $x$, with $\partial _{i}$ the basis vectors of $T_{x}M^{n}$. Then
\begin{eqnarray*}
\langle\hat{\mu}(p), \nu(p) \rangle_{L}=0 \Leftrightarrow \mu^i(x)
u_i(x,t)=0.
\end{eqnarray*}

\end{lemma}

\begin{proof}
Following the convention on the usage of notations in Section
\ref{se2}, let $\overline{\nabla}$ be the covariant connection of
$\mathbb{R}^{n+1}_{1}$. On one hand,
\begin{eqnarray*}
h_{ij}=\langle X_{ij},\nu\rangle_{L}=\langle
\overline{\nabla}_{ij}X, \nu\rangle_{L} =\langle
\overline{\nabla}_{\partial_i}\partial_{j}
+u_i\overline{\nabla}_{\partial_j}\partial_{r}+u_j\overline{\nabla}_{\partial_i}\partial_{r}+u_iu_j
\overline{\nabla}_{\partial_r}\partial_{r}, \nu\rangle_{L}.
\end{eqnarray*}
On the other hand, it is easy to obtain
\begin{eqnarray*}
\overline{\Gamma}_{ij}^{k}=\Gamma_{ij}^{k}, \quad
\overline{\Gamma}_{ij}^{k}=u\sigma_{ij},\quad
\overline{\Gamma}_{i0}^{k}=\frac{1}{u}\delta_{i}^{k}, \quad
\overline{\Gamma}_{0i}^{0}=\overline{\Gamma}_{00}^{i}=\overline{\Gamma}_{00}^{0}=0,
\end{eqnarray*}
where $\Gamma_{ij}^{k}$ denote the Christoffel symbols of
$\mathscr{H}^{n}(1)$ w.r.t. the basis
$\{\partial_{1},\partial_{2},\cdots,\partial_{n}\}$ and
$\overline{\Gamma}_{IJ}^{K}$ stand for the Christoffel symbols of
$\mathbb{R}_{1}^{n+1}$ w.r.t. the basis
$\{\partial_{0}=\partial_{r},\partial_{1},\cdots,\partial_{n}\}$.
Then, these formulae can be verified by direct calculation.
\end{proof}

Using techniques as in Ecker \cite{Eck} (see also \cite{Ge90, Ge3,
Mar}), the problem \eqref{Eq} is degenerated into solving the
following scalar equation with the corresponding initial data and
the corresponding NBC
\begin{equation}\label{Eq-}
\left\{
\begin{aligned}
&\frac{\partial u}{\partial t}=-\frac{v}{H} \qquad &&~\mathrm{in}~
M^n\times(0,\infty)\\
&D_{\mu} u=0  \qquad&&~\mathrm{on}~ \partial M^n\times(0,\infty)\\
&u(\cdot,0)=u_{0} \qquad &&~\mathrm{in}~M^n.
\end{aligned}
\right.
\end{equation}
By Lemma \ref{lemma2-1}, define a new function $\varphi(x,t)=\log
u(x, t)$ and then the mean curvature can be rewritten as
\begin{eqnarray*}
H=\sum_{i=1}^{n}h^i_i=\frac{e^{-\varphi}}{
v}\bigg(n+(\sigma^{ij}+\frac{\varphi^{i}\varphi^{j}}{v^{2}})\varphi_{ij}\bigg).
\end{eqnarray*}
Hence, the evolution equation in \eqref{Eq-} can be rewritten as
\begin{eqnarray*}
\frac{\partial}{\partial t}\varphi=- (1-|D\varphi|^2)\frac{1}
{[n+(\sigma^{ij}+\frac{\varphi^{i}
\varphi^{j}}{v^2})\varphi_{ij}]}:=Q(D\varphi, D^2\varphi).
\end{eqnarray*}
In particular,
 \begin{eqnarray*}
\left(n+(\sigma^{ij}+\frac{\varphi_0^{i}
\varphi_0^{j}}{v^2})\varphi_{0,ij}\right)
 \end{eqnarray*}
 is positive on $M^n$, since $M_0$ is strictly mean convex.
Thus, the problem \eqref{Eq} is again reduced to solve the following
scalar equation with the NBC and the initial data
\begin{equation}\label{Evo-1}
\left\{
\begin{aligned}
&\frac{\partial \varphi}{\partial t}=Q(D\varphi, D^{2}\varphi) \quad
&& \mathrm{in} ~M^n\times(0,T)\\
&D_{\mu} \varphi =0  \quad && \mathrm{on} ~ \partial M^n\times(0,T)\\
&\varphi(\cdot,0)=\varphi_{0} \quad && \mathrm{in} ~ M^n,
\end{aligned}
\right.
\end{equation}
where
$$\left(n+(\sigma^{ij}+\frac{\varphi_0^{i}
\varphi_0^{j}}{v^2})\varphi_{0,ij}\right)$$ is positive on $M^n$.
Clearly, for the initial spacelike graphic hypersurface $M_{0}^{n}$,
$$\frac{\partial Q}{\partial \varphi_{ij}}\Big{|}_{\varphi_0}=\frac{1}{u^{2}H^{2}}(\sigma^{ij}+\frac{\varphi_0^{i}
\varphi_0^{j}}{v^2})$$ is positive on $M^n$. Based on the above
facts, as in \cite{Ge90, Ge3, Mar}, we can get the following
short-time existence and uniqueness for the parabolic system
\eqref{Eq}.

\begin{lemma}
Let $X_0(M^n)=M_{0}^{n}$ be as in Theorem \ref{main1.1}. Then there
exist some $T>0$, a unique solution  $u \in
C^{2+\gamma,1+\frac{\gamma}{2}}(M^n\times [0,T]) \cap C^{\infty}(M^n
\times (0,T])$, where $\varphi(x,t)=\log u(x,t)$, to the parabolic
system \eqref{Evo-1} with the matrix
 \begin{eqnarray*}
\left(n+(\sigma^{ij}+\frac{\varphi^{i}
\varphi^{j}}{v^2})\varphi_{ij}\right)
 \end{eqnarray*}
positive on $M^n$. Thus there exists a unique map $\psi:
M^n\times[0,T]\rightarrow M^n$ such that $\psi(\partial M^n
,t)=\partial M^n$ and the map $\widehat{X}$ defined by
\begin{eqnarray*}
\widehat{X}: M^n\times[0,T)\rightarrow \mathbb{R}^{n+1}_{1}:
(x,t)\mapsto X(\psi(x,t),t)
\end{eqnarray*}
has the same regularity as stated in Theorem \ref{main1.1} and is
the unique solution to the parabolic system \eqref{Eq}.
\end{lemma}

Let $T^{\ast}$ be the maximal time such that there exists some
 \begin{eqnarray*}
u\in C^{2+\gamma,1+\frac{\gamma}{2}}(M^n\times[0,T^{\ast}))\cap
C^{\infty}(M^n\times(0,T^{\ast}))
 \end{eqnarray*}
  which solves \eqref{Evo-1}. In the
sequel, we shall prove a priori estimates for those admissible
solutions on $[0,T]$ where $T<T^{\ast}$.

\section{$C^0$, $\dot{\varphi}$ and gradient estimates} \label{se4}

\begin{lemma}[\bf$C^0$ estimate]\label{lemma3.1}
Let $\varphi$ be a solution of \eqref{Evo-1}, we have
\begin{equation*}
c_1\leq u(x, t) \Theta^{-1}(t, c) \leq c_2, \qquad\quad \forall~
x\in M^n, \ t\in[0,T]
\end{equation*}
for some positive constants $c_{1}$, $c_{2}$, where $\Theta(t,
c):=e^{-\frac{t}{n}+c}$ with
 \begin{eqnarray*}
\inf_{M^n}\varphi(\cdot,0)\leq c\leq \sup_{M^n} \varphi(\cdot,0).
\end{eqnarray*}
\end{lemma}

\begin{proof}
Let $\varphi(x, t)=\varphi(t)$ (independent of $x$) be  the solution
of \eqref{Evo-1} with $\varphi(0)=c$. In this case, the first
equation in \eqref{Evo-1} reduces to an ODE
\begin{eqnarray*}
\frac{d}{d t}\varphi=-\frac{1}{n}.
\end{eqnarray*}
Therefore,
\begin{eqnarray}\label{blow}
\varphi(t)=-\frac{1}{n}t+c.
\end{eqnarray}
Using the maximum principle, we can obtain that
\begin{equation}\label{C^0}
-\frac{1}{n}t+\varphi_{1}\leq\varphi(x, t)
\leq-\frac{1}{n}t+\varphi_{2},
\end{equation}
where $\varphi_1:=\inf_{M^n}\varphi(\cdot,0)$ and
$\varphi_2:=\sup_{M^n} \varphi(\cdot,0)$. The estimate is obtained
since $\varphi =\log u$.
\end{proof}

\begin{lemma}[\bf$\dot{\varphi}$ estimate]\label{lemma3.2}
Let $\varphi$ be a solution of \eqref{Evo-1} and $\Sigma^n$ be the
boundary of a smooth, convex domain defined as in Theorem
\ref{main1.1}, then
\begin{eqnarray*}
\inf_{M^n}\dot{\varphi}(\cdot, 0) \leq \dot{\varphi}(x, t)\leq
\sup_{M^{n}}\dot{\varphi}(\cdot, 0).
\end{eqnarray*}
\end{lemma}

\begin{proof}
Set
\begin{eqnarray*}
\mathcal{M}(x,t)=\dot{\varphi}(x, t).
\end{eqnarray*}
Differentiating both sides of the first evolution equation of
\eqref{Evo-1}, it is easy to get that
\begin{equation} \label{3.4}
\left\{
\begin{aligned}
&\frac{\partial\mathcal{M}}{\partial t}=
Q^{ij}D_{ij}\mathcal{M}+Q^{k}D_k \mathcal{M} \quad
&& \mathrm{in} ~M^n\times(0,T)\\
&D_{\mu}\mathcal{M}=0 \quad && \mathrm{on} ~\partial M^n\times(0,T)\\
&\mathcal{M}(\cdot,0)=\dot{\varphi}_0 \quad && \mathrm{on} ~ M^n,
\end{aligned}
\right.
\end{equation}
where $Q^{ij}:=\frac{ \partial Q}{\partial \varphi_{ij}}$
 and $Q^k:=\frac{ \partial Q}{\partial \varphi_{k}}$.
Then the result follows from the maximum principle.
\end{proof}

\begin{lemma}[\bf Gradient estimate]\label{Gradient}
Let $\varphi$ be a solution of \eqref{Evo-1} and $\Sigma^n$ be the
boundary of a smooth, convex domain described as in Theorem
\ref{main1.1}. Then we have,
\begin{equation}\label{Gra-est}
|D\varphi|\leq \sup_{M^n}|D\varphi(\cdot, 0)|<1, \qquad\quad
\forall~ x\in M^n, \ t\in[0,T].
\end{equation}
\end{lemma}

\begin{proof}
Set $\psi=\frac{|D \varphi|^2}{2}$. By differentiating  $\psi$, we
have
\begin{equation*}
\begin{aligned}
\frac{\partial \psi}{\partial t} =\frac{\partial}{\partial
t}\varphi_m \varphi^m = \dot{\varphi}_m\varphi^m =Q_m \varphi^m.
\end{aligned}
\end{equation*}
Then using the evolution equation of $\varphi$ in (\ref{Evo-1})
yields
\begin{eqnarray*}
\frac{\partial \psi}{\partial t}=Q^{ij}\varphi_{ijm} \varphi^m
+Q^k\varphi_{km} \varphi^m.
\end{eqnarray*}
Interchanging the covariant derivatives, we have
\begin{equation*}
\begin{aligned}
\psi_{ij}&=D_j(\varphi_{mi} \varphi^m)\\&=\varphi_{mij} \varphi^m+\varphi_{mi} \varphi^m_j\\
&=(\varphi_{ijm}+R^l_{imj}\varphi_{l})\varphi^m+\varphi_{mi}\varphi^m_j.
\end{aligned}
\end{equation*}
Therefore, we can express $\varphi_{ijm} \varphi^m$ as
\begin{eqnarray*}
\varphi_{ijm} \varphi^m =\psi_{ij}-R^l_{imj}\varphi_l
\varphi^m-\varphi_{mi} \varphi^m_j.
\end{eqnarray*}
Then, in view of the fact
$R_{ijml}=\sigma_{il}\sigma_{jm}-\sigma_{im}\sigma_{jl}$ on
$\mathscr{H}^{n}(1)$, we have
\begin{equation}\label{gra}
\begin{aligned}
\frac{\partial \psi}{\partial t}&=Q^{ij}\psi_{ij}+Q^k \psi_k
-Q^{ij}(\varphi_i
\varphi_j-\sigma_{ij}|D\varphi|^2)\\&-Q^{ij}\varphi_{mi}
\varphi^{m}_{j}.
\end{aligned}
\end{equation}

Since the matrix $Q^{ij}$ is positive definite, the third and the
fourth terms in the RHS of \eqref{gra} are non-positive. Since $M^n$
is convex, using a similar argument to the proof of \cite[Lemma
5]{Mar} (see page 1308) implies that
\begin{eqnarray*}
D_{\mu}\psi=-\sum\limits_{i,j=1}^{n-1}h_{ij}^{\partial
M^{n}}D_{e_i}\varphi D_{e_j}\varphi \leq
0~~~~\qquad\mathrm{on}~\partial M^n\times(0,T),
\end{eqnarray*}
where an orthonormal frame at $x\in\partial M^{n}$, with
$e_{1},\ldots,e_{n-1}\in T_{x}\partial M^{n}$ and $e_{n}:=\mu$, has
been chosen for convenience in the calculation, and
$h_{ij}^{\partial M^{n}}$ is the second fundamental form of the
boundary $\partial M^{n}\subset\Sigma^{n}$.
 So, we can get
\begin{equation*}
\left\{
\begin{aligned}
&\frac{\partial \psi}{\partial t}\leq Q^{ij}\psi_{ij}+Q^k\psi_k
\qquad &&\mathrm{in}~
M^n\times(0,T)\\
&D_{\mu} \psi \leq 0   && \mathrm{on}~\partial M^n\times(0,T)\\
&\psi(\cdot,0)=\frac{|D\varphi(\cdot,0)|^2}{2} \qquad
&&\mathrm{in}~M^n.
\end{aligned}\right .\end{equation*}
Using the maximum principle, we have
\begin{equation*}
|D\varphi|\leq \sup_{M^n}|D\varphi(\cdot, 0)|,
\end{equation*}
Since $G_{0}:=\{\left(x,u(x,0)\right)|x\in M^{n}\}$ is a spacelike
graph of $\mathbb{R}^{n+1}_{1}$, so we have
\begin{equation*}
|D\varphi|\leq \sup_{M^n}|D\varphi(\cdot, 0)|<1, \qquad\quad
\forall~ x\in M^n, \ t\in[0,T].
\end{equation*}
Our proof is finished.
\end{proof}

\begin{remark}
\rm{The gradient estimate in Lemma \ref{Gradient} makes sure that
the evolving graphs $G_{t}:=\{\left(x,u(x,t)\right)|x\in M^{n},0\leq
t\leq T\}$ are spacelike graphs.}
\end{remark}

Combing the gradient estimate with $\dot{\varphi}$ estimate, we can
obtain
\begin{corollary}
If $\varphi$ satisfies \eqref{Evo-1}, then we have
\begin{eqnarray}\label{w-ij}
0<c_3\leq  H\Theta \leq c_4<+\infty,
\end{eqnarray}
where $c_3$ and $c_4$ are positive constants independent of
$\varphi$.
\end{corollary}

\section{H\"{o}lder Estimates and Convergence} \label{se5}

Set $\Phi=\frac{1}{H}$, $w=\langle X, \nu\rangle_{L}$ and $\Psi=
\frac{\Phi}{w}$. We can get the following evolution equations.

\begin{lemma}\label{EVQ}
Under the assumptions of Theorem \ref{main1.1}, we have
\begin{eqnarray*}
\frac{\partial}{\partial t}g_{ij}=-2\Phi h_{ij},
\end{eqnarray*}
\begin{eqnarray*}
\frac{\partial}{\partial t}g^{ij}=2\Phi h^{ij},
\end{eqnarray*}
\begin{equation*}
\frac{\partial}{\partial t}\nu=\nabla \Phi,
\end{equation*}
\begin{equation*}
\partial_{t}h_{i}^{j}-\Phi H^{-1}\Delta h_{i}^{j}=-\Phi^{2} |A|^2 h_{i}^{j}-2\Phi^{3} H_i H^j
+2\Phi h_{ik}h^{kj},
\end{equation*}
and
\begin{equation}\label{div-for-1}
 \frac{\partial \Psi }{\partial t}=\mathrm{div}_g (H^{-2} \nabla \Psi)-2H^{-2}\Psi^{-1} |\nabla \Psi|^2.
\end{equation}
\end{lemma}

\begin{proof}
It is easy to get the first three evolution equations, and we omit
here. If one really wants to get some hints, we suggest \cite[Lemma
3.1]{GaoY} where the derivation of several evolutions of geometric
quantities (such as the metric, the timelike unit normal vector, the
mean curvature, etc) of spacelike graphic hypersurfaces in Lorentz
manifold along a curvature flow have been shown in details.

 Using the Gauss formula (\ref{Gauss for}),
we have
\begin{equation*}
\begin{split}
\partial_{t}h_{ij}&=
\partial_{t}\langle \partial_i \partial_j X, \nu\rangle_{L}\\
&=\langle \partial_i \partial_j (\Phi \nu), \nu\rangle_{L}
+\langle \Gamma_{ij}^{k}\partial_{k}X-h_{ij}\nu, \partial_{t}\nu\rangle_{L}\\
&=-\partial_i \partial_j \Phi+\Phi\langle\partial_i \partial_j \nu,
\nu\rangle_{L}
+\Gamma_{ij}^{k}\Phi_k\\
&=-\nabla^{2}_{ij} \Phi-\Phi\langle\partial_i (h_{j}^{k}\partial_{k}X), \nu\rangle_{L}\\
&=-\nabla^{2}_{ij} \Phi-\Phi h_{ik}h_{j}^{k}.
\end{split}
\end{equation*}
Direct calculation results in
\begin{equation*}
\nabla^{2}_{ij}\Phi=\Phi(-\frac{1}{H}H_{ij}+\frac{2H_i H_j}{H^2}).
\end{equation*}
By (\ref{Laplace}), we have
\begin{eqnarray*}
\Delta h_{ij}=H_{ij}-H h_{ik}h^{k}_{j}+h_{ij}|A|^2,
\end{eqnarray*}
so
\begin{equation*}
\nabla^{2}_{ij}\Phi=-H^{-2}\Delta h_{ij}-H^{-1} h_{ik}h^{k}_{j}
+H^{-2}|A|^2 h_{ij}+\frac{2H_i H_j}{H^3}.
\end{equation*}
Thus,
\begin{equation*}
\partial_{t}h_{ij}-H^{-2}\Delta h_{ij}=-H^{-2}|A|^2 h_{ij}-\frac{2}{H^3}H_i H_j.
\end{equation*}
Then
\begin{equation*}
\begin{split}
\partial_tH&=  \partial_t g^{ij} h_{ij}+ g^{ij} \partial_t h_{ij}\\
&= 2\Phi h^{ij}h_{ij} +g^{ij} \left( H^{-2}\Delta h_{ij}-H^{-2}|A|^2 h_{ij}-\frac{2}{H^3} \nabla_ iH \nabla_j H\right)\\
&=H^{-1} |A|^2+H^{-2} \Delta H-2 H^{-3} |\nabla H|^2.
\end{split}
\end{equation*}
Clearly,
\begin{eqnarray*}
\partial_{t}w= -\Phi  -  H^{-2} \nabla^i H\langle X,X_i\rangle_{L},
\end{eqnarray*}
using the Weingarten equation (\ref{Wein for}), we have
\begin{eqnarray*}
w_i=-h_{i}^{k}\langle X, X_k\rangle_{L},
\end{eqnarray*}
\begin{eqnarray*}
w_{ij}=-h_{i,j}^{k}\langle X,
X_k\rangle_{L}-h_{ij}+h_{i}^{k}h_{kj}\langle X, \nu\rangle_{L}
=-h_{ij, k}\langle X, X^k\rangle_{L}-h_{ij}+h_{i}^{k}h_{kj}\langle
X, \nu\rangle_{L}.
\end{eqnarray*}
Thus,
\begin{eqnarray*}
\Delta w= -H-\nabla^i H\langle X, X_i\rangle_{L}+|A|^2 \langle X,
\nu\rangle_{L}.
\end{eqnarray*}
 and
\begin{eqnarray*}
\partial_t w= H^{-2} \Delta w-H^{-2} w |A|^2.
\end{eqnarray*}
Hence
\begin{equation*}
\begin{split}
\frac{\partial \Psi }{\partial t}&= -  \frac{1}{ H^2} \frac{1}{w} \partial_tH -\frac{1}{ H} \frac{1}{w^2} \partial_tw\\
&= 2 H^{-5}w^{-1} |\nabla H|^2
 -H^{-4}w^{-1} \Delta H -  H^{-3}w^{-2} \Delta w.
\end{split}
\end{equation*}
In order to prove (\ref{div-for-1}), we calculate
$$\nabla_i \Psi =- H^{-2} w^{-1} \nabla_i H-  H^{-1} w^{-2} \nabla_i w.$$
and
\begin{equation*}
\begin{split}
\nabla^2_{ij}\Psi&= 2 H^{-3} w^{-1} \nabla_i H \nabla_j H +  H^{-2} w^{-2} \nabla_i H \nabla_j w-  H^{-2} w^{-1}\nabla^2_{ij} H \\
& \quad +  H^{-2} w^{-2} \nabla_i w \nabla_j H+ 2 H^{-1} w^{-3}
\nabla_i w \nabla_j w - H^{-1} w^{-2} \nabla^2_{ij} w.
\end{split}
\end{equation*}
Thus
\begin{equation*}
\begin{split}
 H^{-2} \Delta \Psi
&= 2 H^{-5} w^{-1} |\nabla H|^2+2 H^{-3} w^{-3} |\nabla w|^2+2 H^{-4} w^{-2} \nabla_i H \nabla^i w  \\
&- H^{-4} w^{-1}\Delta H- H^{-3} w^{-2} \Delta w.
\end{split}
\end{equation*}
So we have
\begin{equation*}
\begin{split}
&\mbox{div} ( H^{-2} \nabla \Psi)=- 2  H^{-3} \nabla_i \Psi \nabla^i H+  H^{-2} \Delta \Psi\\
&=4  H^{-5} w^{-1} |\nabla H|^2+ 4 H^{-4} w^{-2} \nabla_i w \nabla^i H +2 H^{-3} w^{-3} |\nabla w|^2\\
&- H^{-4} w^{-1}\Delta H- H^{-3} w^{-2} \Delta w.
\end{split}
\end{equation*}
and
\begin{equation*}
2H^{-1} w |\nabla \Psi|^2= 2 H^{-5} w^{-1} |\nabla H|^2 + 2 H^{-3}
w^{-3} |\nabla w|^2+4 H^{-4} w^{-2} \nabla_i H \nabla^i w.
\end{equation*}
As above, we have
\begin{equation*}
\frac{\partial \Psi }{\partial t}-\mbox{div} ( H^{-2} \nabla
\Psi)+2H^{-1} w |\nabla \Psi|^2=0.
\end{equation*}
The proof is finished.
\end{proof}

Now, we define the  rescaled flow by
\begin{equation*}
\widetilde{X}=X\Theta^{-1}.
\end{equation*}
Thus,
\begin{equation*}
\widetilde{u}=u\Theta^{-1},
\end{equation*}
\begin{equation*}
\widetilde{\varphi}=\varphi-\log\Theta,
\end{equation*}
and the rescaled Gauss curvature is given by
\begin{equation*}
\widetilde{H}=H\Theta.
\end{equation*}
Then, the rescaled scalar curvature equation takes the form
\begin{equation*}
\frac{\partial}{\partial
t}\widetilde{u}=-\frac{v}{\widetilde{H}}+\frac{1}{n}\widetilde{u}.
\end{equation*}
 Then $\widetilde{u}$ satisfies
\begin{equation}\label{Eq-re}
\left\{
\begin{aligned}
&\frac{\partial}{\partial
t}\widetilde{u}=-\frac{v}{\widetilde{H}}+\frac{1}{n}\widetilde{u}
\qquad && \mathrm{in}~
M^n\times(0,T)\\
&D_{\mu} \widetilde{u}=0  \qquad && \mathrm{on}~ \partial M^n\times(0,T)\\
&\widetilde{u}(\cdot,0)=\widetilde{u}_{0}  \qquad &&
\mathrm{in}~M^n.
\end{aligned}
\right.
\end{equation}

\begin{lemma}\label{res-01}
Let $X$ be a solution of (\ref{Eq}) and $\widetilde{X}=X
\Theta^{-1}$ be the rescaled solution. Then
\begin{equation*}
\begin{split}
&D \widetilde{u}=D u \Theta^{-1}, ~~~~D \widetilde{\varphi}=D \varphi,~~~~ \frac{\partial \widetilde{u}}{\partial t}=\frac{ \partial u}{\partial t} \Theta^{-1}+ \frac{1}{n}u\Theta^{-1},\\
&\widetilde{g}_{ij}=
\Theta^{-2}g_{ij},~~~~\widetilde{g}^{ij}=\Theta^{2}
g^{ij},~~~~\widetilde{h}_{ij}=h_{ij}\Theta^{-1}.
\end{split}
\end{equation*}
\end{lemma}
\begin{proof}
These relations can be computed directly.
\end{proof}

\begin{lemma} \label{lemma4-3}
Let $u$ be a solution to the parabolic system \eqref{Evo-1}, where
$\varphi(x,t)=\log u(x,t)$, and $\Sigma^n$ be the boundary of a
smooth, convex domain described as in Theorem \ref{main1.1}. Then
there exist some $\beta>0$ and some $C>0$ such that the rescaled
function $\widetilde{u}(x,t):=u(x,t) \Theta^{-1}(t)$ satisfies
\begin{equation}
 [D \widetilde{u}]_{\beta}+\left[\frac{\partial \widetilde{u}}{\partial t}\right]_{\beta}+[\widetilde{H}]_{\beta}\leq C(||u_{0}||_{C^{2+\gamma,1+\frac{\gamma}{2}}(M^n)}, n, \beta, M^n),
\end{equation}
where $[f]_{\beta}:=[f]_{x,\beta}+[f]_{t,\frac{\beta}{2}}$ is the
sum of the H\"{o}lder coefficients of $f$ in $M^n\times[0,T]$ with
respect to $x$ and $t$.
\end{lemma}

\begin{proof}
We divide our proof in  three steps\footnote{In the proof of Lemma
\ref{lemma4-3}, the constant $C$ may differ from each other.
However, we abuse the symbol $C$ for the purpose of convenience.}.

\textbf{Step 1:} We need to prove that
\begin{equation*}
  [D \widetilde{u}]_{x,\beta}+[D \widetilde{u}]_{t,\frac{\beta}{2}}\leq C(|| u_0||_{ C^{2+\gamma,1+\frac{\gamma}{2}}(M^n)}, n, \beta, M^n).
\end{equation*}
According to Lemmas \ref{lemma3.1}, \ref{lemma3.2} and
\ref{Gradient}, it follows that
$$|D \widetilde{u}|+\left|\frac{\partial \widetilde{u}}{\partial t}\right|\leq C(|| u_0||_{ C^{2+\gamma,1+\frac{\gamma}{2}}(M^n)}, M^n).$$
Then we can easily obtain the bound of $[\widetilde{u}]_{x,\beta}$
and $[\widetilde{u}]_{t,\frac{\beta}{2}}$ for any $0<\beta<1$. Lemma
3.1 in \cite[Chap. 2]{La2} implies that the bound for
$[D\widetilde{u}]_{t,\frac{\beta}{2}}$ follows from a bound for
$[\widetilde{u}]_{t,\frac{\beta}{2}}$ and
$[D\widetilde{u}]_{x,\beta}$. Hence it remains to bound  $[D
\varphi]_{x,\beta}$ since $D\widetilde{u}=\widetilde{u} D\varphi$.
For this,  fix $t$ and the equation (\ref{Evo-1}) can be rewritten
as an elliptic Neumann problem
 \begin{equation} \label{key1}
   -\mbox{div}_{\sigma}\left(\frac{D \varphi}{\sqrt{1-|D\varphi|^2}}\right)=\frac{n}{\sqrt{1-|D\varphi|^2}}+
    \frac{\sqrt{1-|D\varphi|^2}}{\dot{\varphi}}.
 \end{equation}
In fact, the equation (\ref{key1}) is of the form
$D_i(a^i(p))+a(x,t)=0$. Since $\dot{\varphi}$ and $|D\varphi|$ are
bounded, we know $a$ is a bounded function in $x$ and $t$. We define
$a^{ij}(p):=\frac{\partial a^i}{\partial p^j}$, the smallest and
largest eigenvalues of $a^{ij}$ are controlled due to the estimate
for $|D\varphi|$. By \cite[Chap. 3; Theorem 14.1; Chap. 10, \S
2]{La1}, we have the interior estimate and boundary estimate of
$[D\varphi]_{x,\beta}$. For a detail proof of this estimate, readers
are referred to \cite{Mar1}.

\textbf{Step 2:} The next thing to do is to show that
\begin{equation*}
  \left[\frac{\partial \widetilde{u}}{\partial t}\right]_{x,\beta}+\left[\frac{\partial \widetilde{u}}{\partial t}\right]_{t,\frac{\beta}{2}}\leq
  C(||u_0||_{ C^{2+\gamma,1+\frac{\gamma}{2}}(M^n)}, n, \beta,
  M^n).
\end{equation*}
  As $\frac{\partial}{\partial t}\widetilde{u}=\widetilde{u}\left(\frac{v}{\widetilde{u}\widetilde{H}}+\frac{1}{n}\right)$, it is enough to bound
  $\left[\frac{v}{\widetilde{u} \widetilde{H}}\right]_{\beta}$.
Set $\widetilde{w}(t):=\frac{v}{\widetilde{u} \widetilde{H}}=
\Psi$. Let $\widetilde{\nabla}$ be the Levi-Civita connection of
$\widetilde{M}_{t}:=\widetilde{X}(M^n,t)$ w.r.t. the metric
$\widetilde{g}$. Combining (\ref{div-for-1}) with Lemma
\ref{res-01}, we get
\begin{equation}\label{div-form-02}
\frac{\partial \widetilde{w}}{\partial t}
=\mbox{div}_{\widetilde{g}} ( \widetilde{H}^{-2} \widetilde{\nabla}
\widetilde{w})-2 \widetilde{H}^{-2} \widetilde{w}^{-1}
|\widetilde{\nabla} \widetilde{w}|^2_{\widetilde{g}}.
\end{equation}

The weak formulation of (\ref{div-form-02}) is
\begin{equation}\label{div-form-03}
\begin{split}
\int_{t_0}^{t_1} \int_{\widetilde{M}_t}  \frac{\partial
\widetilde{w} }{\partial t}  \eta d\mu_t dt & =\int_{t_0}^{t_1}
\int_{\widetilde{M}_t} \mbox{div}_{\widetilde{g}} (
\widetilde{H}^{-2} \widetilde{\nabla} \widetilde{w}) \eta -2
\widetilde{H}^{-2} \widetilde{w}^{-1} |\widetilde{\nabla}
\widetilde{w}|^2_{\widetilde{g}} \eta d\mu_t dt.
\end{split}
\end{equation}
Since $\nabla_{\mu} \widetilde{\varphi}=0$, the boundary integrals
all vanish, the interior and boundary estimates are basically the
same. We define the test function $\eta:=\xi^2 \widetilde{w}$, where
$\xi$ is a smooth function with values in $[0,1]$ and is supported
in a small parabolic neighborhood. Then
\begin{equation}\label{imcf-hec-for-02}
\begin{split}
\int_{t_0}^{t_1} \int_{\widetilde{M}_t}  \frac{\partial
\widetilde{w} }{\partial t}  \xi^2 \widetilde{w} d\mu_t dt=
\frac{1}{2}||\widetilde{w} \xi||_{2,\widetilde{M}_t}^2
\Big{|}_{t_0}^{t_1}-\int_{t_0}^{t_1} \int_{\widetilde{M}_t}  \xi
\dot{\xi} \widetilde{w}^2 d\mu_t dt.
\end{split}
\end{equation}
Using the divergence theorem and Young's inequality, we can obtain
\begin{equation}\label{imcf-hec-for-03}
\begin{split}
&\int_{t_0}^{t_1} \int_{\widetilde{M}_t}  \mbox{div}_{\widetilde{g}} ( \widetilde{H}^{-2} \widetilde{\nabla} \widetilde{w})  \xi^2 \widetilde{w}  d\mu_tdt\\
&=-\int_{t_0}^{t_1} \int_{\widetilde{M}_t}   \widetilde{H}^{-2}
\xi^2\widetilde{\nabla}_i \widetilde{w}
\widetilde{\nabla}^i\widetilde{w}  d\mu_tdt
-2\int_{t_0}^{t_1} \int_{\widetilde{M}_t} \widetilde{H}^{-2} \xi \widetilde{w}\widetilde{\nabla}_i\widetilde{w} \widetilde{\nabla}^i \xi d\mu_tdt\\
&\leq\int_{t_0}^{t_1} \int_{\widetilde{M}_t}   \widetilde{H}^{-2}
|\widetilde{\nabla} \xi|^2\widetilde{w}^2  d\mu_tdt.
\end{split}
\end{equation}
Combing (\ref{imcf-hec-for-02}) and (\ref{imcf-hec-for-03}), we have
 \begin{equation}\label{imcf-hec-for-04}
\begin{split}
&\frac{1}{2}||\widetilde{w} \xi||_{2,\widetilde{M}_t}^2
\Big{|}_{t_0}^{t_1}
+2\int_{t_0}^{t_1} \int_{\widetilde{M}_t} \widetilde{H}^{-2} |\widetilde{\nabla} \widetilde{w}|^2  \xi^2   d\mu_tdt\\
& \leq \int_{t_0}^{t_1} \int_{\widetilde{M}_t}  \xi |\dot{\xi}|
\widetilde{w}^2 d\mu_t dt +\int_{t_0}^{t_1} \int_{\widetilde{M}_t}
\widetilde{H}^{-2} |\widetilde{\nabla} \xi|^2\widetilde{w}^2
d\mu_tdt,
\end{split}
\end{equation}
 which implies that
 \begin{equation}\label{imcf-hec-for-06}
\begin{split}
&\frac{1}{2}||\widetilde{w} \xi||_{2,\widetilde{M}_t}^2
\Big{|}_{t_0}^{t_1}
+\frac{2}{\max(\widetilde{H}^{2}) }\int_{t_0}^{t_1} \int_{\widetilde{M}_t} |\widetilde{\nabla} \widetilde{w}|^2  \xi^2   d\mu_tdt\\
& \leq \left(1+ \frac{1}{\min( \widetilde{H}^{2})}\right)
\int_{t_0}^{t_1} \int_{\widetilde{M}_t}  \widetilde{w}^2 (\xi
|\dot{\xi}| +|\widetilde{\nabla} \xi|^2)d\mu_t dt.
\end{split}
\end{equation}
This means that $\widetilde{w}$ belong to the De Giorgi class of
functions in $M^n \times [0,T)$. Similar to the arguments in
\cite[Chap. 5, \S 1 and \S 7]{La2},  there exist  constants $\beta$
and $C$ such that
$$[\widetilde{w}]_{\beta}\leq C ||\widetilde{w}||_{L^{\infty}(M^n \times [0,T))}\leq  C(|| u_0||_{ C^{2+\gamma,1+\frac{\gamma}{2}}(M^n)}, n, \beta, M^n).$$

\textbf{Step 3:} Finally, we have to show that
\begin{equation*}
  [ \widetilde{H}]_{x,\beta}+[\widetilde{H}]_{t,\frac{\beta}{2}}\leq C(|| u_0||_{ C^{2+\gamma,1+\frac{\gamma}{2}}(M^n)}, n, \beta, M^n).
\end{equation*}
This follows from the fact that
$$\widetilde{H}=\frac{\sqrt{1-|D\varphi|^2}}{\widetilde{u} \widetilde{w}}$$
together with the estimates for $\widetilde{u}$, $\widetilde{w}$,
$D\varphi$.
\end{proof}

Then we can obtain the following higher-order estimates:
\begin{lemma}
Let $u$ be a solution to the parabolic system \eqref{Evo-1}, where
$\varphi(x,t)=\log u(x,t)$, and $\Sigma^n$ be the boundary of a
smooth, convex domain described as in Theorem \ref{main1.1}. Then
for any $t_0\in (0,T)$ there exist some $\beta>0$ and some $C>0$
such that
\begin{equation}\label{imfcone-holder-01}
||\widetilde{u}||_{C^{2+\beta,1+\frac{\beta}{2}}(M^n\times
[0,T])}\leq C(|| u_0||_{ C^{2+\gamma, 1+\frac{\gamma}{2}}(M^n)}, n,
\beta, M^n)
\end{equation}
and for all $k\in \mathbb{N}$,
\begin{equation}\label{imfcone-holder-02}
||\widetilde{u}||_{C^{2k+\beta,k+\frac{\beta}{2}}(M^n\times
[t_0,T])}\leq C(||u_0(\cdot,
t_0)||_{C^{2k+\beta,k+\frac{\beta}{2}}(M^n)}, n, \beta, M^n).
\end{equation}
\end{lemma}

\begin{proof}
By Lemma \ref{lemma2-1}, we have
$$uvH=n+(\sigma^{ij}+\frac{\varphi^{i}\varphi^{j}}{v^{2}})\varphi_{ij}=n+u^2 \Delta_g \varphi.$$
Since
$$u^2 \Delta_g \varphi=\widetilde{u}^2 \Delta_{\widetilde{g}} \varphi=
-| \widetilde{\nabla} \widetilde{u}|^2+ \widetilde{u}
\Delta_{\widetilde{g}} \widetilde{u},$$ then
\begin{equation*}
\begin{split}
\frac{\partial \widetilde{u}}{\partial t}&=\frac{ \partial u}{\partial t} \Theta^{-1}+\frac{1}{n} \widetilde{u}\\
&=\frac{uvH}{u H^2} \Theta^{-1} - \frac{2v}{H} \Theta^{-1}+ \frac{1}{n}\widetilde{u}\\
&=\frac{\Delta_{\widetilde{g}} \widetilde{u}}{\widetilde{H}^2}-
\frac{2v}{\widetilde{H}} + \frac{1}{n}\widetilde{u} +\frac{n-|
\widetilde{\nabla} \widetilde{u}|^2}{\widetilde{u} \widetilde{H}^2},
\end{split}
\end{equation*}
which is  a uniformly parabolic equation with H\"{o}lder continuous
coefficients. Therefore, the linear theory (see \cite[Chap.
4]{Lieb}) yields the inequality (\ref{imfcone-holder-01}).

Set $\widetilde{\varphi}=\log \widetilde{u}$, and then the rescaled
version of the evolution equation in (\ref{Eq-re}) takes the form
\begin{equation*}
  \frac{\partial \widetilde{\varphi}}{\partial t}=- \frac{ v^2}{ \left[n+\left(\sigma^{ij}+\frac{\widetilde{\varphi}^i\widetilde{\varphi}^j}{v^2}\right) \widetilde{\varphi}_{ij}\right]}+\frac{1}{n},
\end{equation*}
where $v=\sqrt{1-|D \widetilde{\varphi}|^2}$. According to the
$C^{2+\beta,1+\frac{\beta}{2}}$-estimate of $\widetilde{u}$ (see
Lemma \ref{lemma4-3}), we can treat the equations for $\frac{
\partial\widetilde{\varphi}}{\partial t}$ and $D_i
\widetilde{\varphi}$ as second-order linear uniformly parabolic PDEs
on $M^n\times [t_0,T]$. At the initial time $t_0$, all compatibility
conditions are satisfied and the initial function $u(\cdot,t_0)$ is
smooth. We can obtain a $C^{3+\beta, \frac{3+\beta}{2}}$-estimate
for $D_i \widetilde{\varphi}$ and a $C^{2+\beta,
\frac{2+\beta}{2}}$-estimate for $\frac{
\partial\widetilde{\varphi}}{\partial t}$ (the estimates are
independent of $T$) by Theorem 4.3 and Exercise 4.5 in \cite[Chap.
4]{Lieb}. Higher regularity can be proven by induction over $k$.
\end{proof}

\begin{theorem} \label{key-2}
Under the hypothesis of Theorem \ref{main1.1}, we conclude
\begin{equation*}
T^{*}=+\infty.
\end{equation*}
\end{theorem}
\begin{proof}
The proof of this result is quite similar to the corresponding
argument in \cite[Lemma 8]{Mar} and so is omitted.
\end{proof}

\section{Convergence of the rescaled flow} \label{se6}

We know that after the long-time existence of the flow has been
obtained (see Theorem \ref{key-2}), the rescaled version of the
system (\ref{Evo-1}) satisfies
\begin{equation}
\left\{
\begin{aligned}
&\frac{\partial}{\partial
t}\widetilde{\varphi}=\widetilde{Q}(D\widetilde{\varphi},
D^2\widetilde{\varphi})  \qquad &&\mathrm{in}~
M^n\times(0,\infty)\\
&D_{\mu} \widetilde{\varphi}=0  \qquad &&\mathrm{on}~ \partial M^n\times(0,\infty)\\
&\widetilde{\varphi}(\cdot,0)=\widetilde{\varphi}_{0} \qquad
&&\mathrm{in}~M^n,
\end{aligned}
\right.
\end{equation}
where
$$\widetilde{Q}(D\widetilde{\varphi},
D^2\widetilde{\varphi}):=- \frac{ v^2}{
\left[n+\left(\sigma^{ij}+\frac{\widetilde{\varphi}^i\widetilde{\varphi}^j}{v^2}\right)
\widetilde{\varphi}_{ij}\right]}+\frac{1}{n}$$ and
$\widetilde{\varphi}=\log \widetilde{u}$. Similar to what has been
done in the $C^1$ estimate (see Lemma \ref{Gradient}), we can deduce
a decay estimate of $\widetilde{u}(\cdot, t)$ as follows.

\begin{lemma} \label{lemma5-1}
Let $u$ be a solution of \eqref{Eq-}, then we have
\begin{equation}\label{Gra-est1}
|D\widetilde{u}(x, t)|\leq\lambda\sup_{M^n}|D\widetilde{u}(\cdot,
0)|,
\end{equation}
where $\lambda$ is a positive constant depending only on $c_{1}$,
$c_{2}$.
\end{lemma}

\begin{proof}
Set $\widetilde{\psi}=\frac{|D \widetilde{\varphi}|^2}{2}$. Similar
to that in Lemma \ref{Gradient}, we can obtain
\begin{equation}\label{gra-}
\frac{\partial \widetilde{\psi}}{\partial t}=\widetilde{Q}^{ij}
\widetilde{\psi}_{ij}+\widetilde{Q}^k \widetilde{\psi}_k
-\widetilde{Q}^{ij}(\widetilde{\varphi}_i
\widetilde{\varphi}_j-\sigma_{ij}|D\widetilde{\varphi}|^2)-\widetilde{Q}^{ij}\widetilde{\varphi}_{mi}
\widetilde{\varphi}^{m}_{j},
\end{equation}
 with the boundary condition
\begin{equation*}
\begin{aligned}
D_ \mu \widetilde{\psi}\leq 0.
\end{aligned}
\end{equation*}
So we have
\begin{equation*}
\left\{
\begin{aligned}
&\frac{\partial \widetilde{\psi}}{\partial s}\leq
\widetilde{Q}^{ij}\widetilde{\psi}_{ij}+\widetilde{Q}^k\widetilde{\psi}_k
\quad &&\mathrm{in}~
M^n\times(0,\infty)\\
&D_\mu \widetilde{\psi} \leq 0 \quad &&\mathrm{on}~\partial M^n\times(0,\infty)\\
&\psi(\cdot,0)=\frac{|D\widetilde{\varphi}(\cdot,0)|^2}{2}
\quad&&\mathrm{in}~M^n.
\end{aligned}\right.
\end{equation*}
Using the maximum principle and Hopf's lemma, we can get the
gradient estimate of $\widetilde{\varphi}$, and then the conclusion
(\ref{Gra-est1}) follows.
\end{proof}

\begin{lemma}\label{rescaled flow}
Let $u$ be a solution of the flow \eqref{Eq-}. Then,
\begin{equation*}
\widetilde{u}(\cdot, t)
\end{equation*}
converges to a real number as $t\rightarrow +\infty$.
\end{lemma}
\begin{proof}
Set $f(t):= \mathcal{H}^n(M_{t}^{n})$, which, as before, represents
the $n$-dimensional Hausdorff measure of $M_{t}^{n}$ and is actually
the area of $M_{t}^{n}$. According to the first variation of a
submanifold, see e.g. \cite{ls},
 and the fact $-\mbox{div}_{M_{t}^{n}}\nu=H$, we have
\begin{equation}\label{imcf-crf-for-01}
\begin{split}
f'(t)&=\int_{M_{t}^{n}} \mbox{div}_{M_{t}^{n}} \left( \frac{\nu}{H}\right) d\mathcal{H}^n\\
&=\int_{M_{t}^{n}} \sum_{i=1}^n \left\langle \nabla_{e_i}\left(\frac{\nu}{H}\right), e_i\right\rangle_{L} d\mathcal{H}^n\\
&=-f(t),
\end{split}
\end{equation}
where $\{e_{i}\}_{1\leq i\leq n}$ is some orthonormal basis of the
tangent bundle $TM_{t}^{n}$. By (\ref{imcf-crf-for-01}), we have
$$f(t)=\mathcal{H}^n(M_{0}^{n})e^{-t}.$$
 Therefore, the rescaled
hypersurface $\widetilde{M}_t=M_{t}^{n} \Theta^{-1}$ satisfies the
following inequality
 \begin{eqnarray*}
\frac{\mathcal{H}^n(M_{0}^{n})}{ e^{n \varphi_2}} \leq
\mathcal{H}^n(\widetilde{M}_t)\leq
\frac{\mathcal{H}^n(M_{0}^{n})}{e^{n \varphi_1}},
 \end{eqnarray*}
 which implies that the area
of  $\widetilde{M}_t$ is bounded and the bounds are independent of
$t$. Here $\varphi_1=\inf_{M^n} \varphi(\cdot,0)$ and
$\varphi_2=\sup_{M^n} \varphi(\cdot,0)$. Together with
(\ref{imfcone-holder-01}), Lemma \ref{lemma5-1} and the
Arzel\`{a}-Ascoli theorem, we conclude that $\widetilde{u}(\cdot,t)$
must converge in $C^{\infty}(M^n)$ to a constant function
$r_{\infty}$ with
\begin{eqnarray*}
\frac{1}{e^{\varphi_{2}}}\left(\frac{\mathcal{H}^n(M_{0}^{n})}{\mathcal{H}^n(M^{n})}\right)^{\frac{1}{n}}\leq
r_{\infty}
\leq\frac{1}{e^{\varphi_{1}}}\left(\frac{\mathcal{H}^n(M_{0}^{n})}{\mathcal{H}^n(M^{n})}\right)^{\frac{1}{n}},
\end{eqnarray*}
i.e.,
\begin{eqnarray}\label{radius}
\frac{1}{\sup\limits_{M^{n}}u_{0}}\left(\frac{\mathcal{H}^n(M_{0}^{n})}{\mathcal{H}^n(M^{n})}\right)^{\frac{1}{n}}\leq
r_{\infty}
\leq\frac{1}{\inf\limits_{M^{n}}u_{0}}\left(\frac{\mathcal{H}^n(M_{0}^{n})}{\mathcal{H}^n(M^{n})}\right)^{\frac{1}{n}}.
\end{eqnarray}
This completes the proof.
\end{proof}

So, we have
\begin{theorem}\label{rescaled flow}
The rescaled flow
\begin{equation*}
\frac{d
\widetilde{X}}{dt}=\frac{1}{\widetilde{H}}\nu+\frac{1}{n}\widetilde{X}
\end{equation*}
exists for all time and the leaves converge in $C^{\infty}$ to a
piece of hyperbolic plane of center at origin and radius
$r_{\infty}$, i.e., a piece of $\mathscr{H}^{n}(r_{\infty})$, where
$r_{\infty}$ satisfies (\ref{radius}).
\end{theorem}

\section{Examples} \label{se7}

\subsection{Example in the Lorentz-Minkowski case}

In Theorem \ref{main1.1} here, if $u_{0}=R_{0}$ for some positive
constant $R_{0}>0$, i.e., the initial hypersurface $M^{n}_{0}$ is a
convex piece of the hyperbolic plane $\mathscr{H}^{n}(R_{0})$, then,
in this particular case, the tangential vector on $M^{n}_{t}$
becomes
$$X_{i} = \partial_{i},$$
and the corresponding past-directed timelike unit normal vector is
given by
$$\nu = -\partial_{r}.$$
By direct calculation, the mean curvature of
$\mathscr{H}^{n}(R_{0})$ is given by
$$H=\frac{n}{R_{0}}.$$
One can choose
$$\mu = \frac{1}{\sqrt{n}}\partial_{i}$$
to be the spacelike unit vector, which clearly satisfies the
property required in Theorem \ref{main1.1}. So, the system
(\ref{Eq}) degenerates into the following initial-boundary value
problem (IBVP for short)
\begin{equation*}
\left\{
\begin{aligned}
&\frac{du(t)}{dt} = -\frac{u(t)}{n} \quad &&\mathrm{in}~
M^n\times(0,\infty)\\
&D_\mu u(t) = 0 \quad &&\mathrm{on}~\partial M^n\times(0,\infty)\\
&u(0) = R_{0} \quad&&\mathrm{in}~M^{n},
\end{aligned}\right.
\end{equation*}
which obviously has the solution
\begin{eqnarray*}
u(t) = R_{0}e^{-\frac{1}{n}t}.
\end{eqnarray*}
Therefore, in this setting, the evolving hypersurfaces $M^{n}_{t}$
should be $M^{n}_{t}=\{(x,R_{0}e^{-\frac{1}{n}t})|x\in
M^{n}\subset\mathscr{H}^{n}(1),0\leq t<\infty\}$, and clearly, as
$t\rightarrow\infty$, $u(t)\rightarrow 0$, which implies that
$M^{n}_{t}$ tends to the vertex of the lightcone of
$\mathbb{R}^{n+1}_{1}$ as $t\rightarrow\infty$ (i.e., $M^{n}_{t}$
tends to the origin of $\mathbb{R}^{n+1}$ as $t\rightarrow\infty$).
That is to say, in this setting, although one has the long-time
existence for the IMCF considered in Theorem \ref{main1.1}, the flow
 forms a singularity as $t\rightarrow\infty$.

 \subsection{Example in the Euclidean case}

In \cite[Theorem 1]{Mar}, if  $u_{0}=R_{0}$ for some positive
constant $R_{0}>0$, i.e., the initial hypersurface $M^{n}_{0}$ is a
convex piece of the Euclidean $n$-sphere $\mathbb{S}^{n}(R_{0})$
centered at the origin and with radius $R_{0}$, then, in this
particular case, the system (IMCF) in \cite[Theorem 1]{Mar}
degenerates into the following IBVP
\begin{equation} \label{se7-1}
\left\{
\begin{aligned}
&\frac{du(t)}{dt} = \frac{u(t)}{n} \quad &&\mathrm{in}~
M^n\times(0,\infty)\\
&D_\mu u(t) = 0 \quad &&\mathrm{on}~\partial M^n\times(0,\infty)\\
&u(0) = R_{0} \quad&&\mathrm{in}~M^{n},
\end{aligned}\right.
\end{equation}
where $M^{n}\subset \mathbb{S}^{n}$ is some convex piece of
$\mathbb{S}^{n}$, and $\mu$ is the outward unit normal vector field
of the open smooth convex cone $\Sigma^n:=\{rx\in \mathbb{R}^{n+1}|
r>0, x\in
\partial M^n\}$. Clearly, the IBVP (\ref{se7-1}) has the solution
\begin{eqnarray*}
u(t) = R_{0}e^{\frac{1}{n}t}.
\end{eqnarray*}
Therefore, in this setting, the evolving hypersurfaces $M^{n}_{t}$
should be $M^{n}_{t}=\{(x,R_{0}e^{\frac{1}{n}t})|x\in
M^{n}\subset\mathbb{S}^{n},0\leq t<\infty\}$, and clearly, as
$t\rightarrow\infty$, $u(t)\rightarrow\infty$, which implies that
the IMCF considered in \cite[Theorem 1]{Mar} is an expanding flow.

\vspace {5 mm}

\section*{Acknowledgments}
This work is partially supported by the NSF of China (Grant Nos.
11801496 and 11926352), the Fok Ying-Tung Education Foundation
(China) and  Hubei Key Laboratory of Applied Mathematics (Hubei
University).

\end{document}